\documentclass[a4paper, reqno]{amsart}

\usepackage[UKenglish]{babel}
\usepackage{amsmath, amssymb, amsthm}
\usepackage{scrextend}
\usepackage[hidelinks]{hyperref}
\usepackage{xpatch}
\usepackage{color}
\usepackage{tikz-cd}
\usepackage{enumitem}
\usepackage{theoremref}

\usepackage[inline,final]{showlabels}
\showlabels{thlabel}

\setlist[description]{font=\normalfont\scshape}

\xpatchcmd{\proof}{\itshape}{\normalfont\bfseries}{}{}
\newtheoremstyle{repeat}{}{}{\itshape}{}{\bfseries}{.}{.5em}{#3, repeated}

\newtheorem{theorem}{Theorem}[section]
\newtheorem{proposition}[theorem]{Proposition}
\newtheorem{lemma}[theorem]{Lemma}
\newtheorem{corollary}[theorem]{Corollary}
\newtheorem{fact}[theorem]{Fact}
\newtheorem{claim}{Claim}[theorem]

\theoremstyle{definition}
\newtheorem{definition}[theorem]{Definition}
\newtheorem{remark}[theorem]{Remark}
\newtheorem{convention}[theorem]{Convention}

\theoremstyle{repeat}
\newtheorem*{repeated-theorem}{Repeat}

\newcommand{\MM}{\mathfrak{M}}
\newcommand{\T}{\mathcal{T}}
\newcommand{\Z}{\mathbb{Z}}

\DeclareMathOperator{\tp}{tp}
\DeclareMathOperator{\Lstp}{Lstp}
\DeclareMathOperator{\Aut}{Aut}

\DeclareMathOperator{\dom}{dom}
\renewcommand{\d}{\operatorname{d}}
\DeclareMathOperator{\EM}{EM}

\renewcommand{\phi}{\varphi}
\newcommand{\equivls}{\equiv^\textup{Ls}}

\newcommand{\lex}{{\textup{lex}}}

\def\Ind#1#2{#1\setbox0=\hbox{$#1x$}\kern\wd0\hbox to 0pt{\hss$#1\mid$\hss}
\lower.9\ht0\hbox to 0pt{\hss$#1\smile$\hss}\kern\wd0}
\def\ind{\mathop{\mathpalette\Ind{}}}
\def\Notind#1#2{#1\setbox0=\hbox{$#1x$}\kern\wd0\hbox to 0pt{\mathchardef
\nn="3236\hss$#1\nn$\kern1.4\wd0\hss}\hbox to 0pt{\hss$#1\mid$\hss}\lower.9\ht0
\hbox to 0pt{\hss$#1\smile$\hss}\kern\wd0}

\title{Corrigendum to ``Kim-independence in positive logic''}
\author{Jan Dobrowolski and Mark Kamsma}
\date{\today}
\thanks{The second author is supported by the EPSRC grant EP/X018997/1.}
\email[Jan Dobrowolski]{dobrowol@math.uni.wroc.pl}
\address[Jan Dobrowolski]{Department of Mathematics, University of Manchester, Oxford Road, Manchester, M13 9PL, UK and \newline
Instytut Matematyczny Uniwersytetu Wroc\l{}awskiego, pl. Grunwaldzki 2/4, 50-383 Wroc\l{}aw}
\email[Mark Kamsma]{mark@markkamsma.nl}
\urladdr[Mark Kamsma]{https://markkamsma.nl}
\address[Mark Kamsma]{School of Mathematical Sciences, Queen Mary University of London, London, E1 4NS, UK}
 
\begin{document}

\begin{abstract}
The proof of the Independence Theorem for Kim-independence in positive thick NSOP$_1$ theories from \cite{dobrowolski_kim-independence_2022} contains a gap. The theorem is still true, and in this corrigendum we give a different proof.
\end{abstract}

% Title
\maketitle

% Table of contents
\tableofcontents

\section{Introduction}
\label{sec:introduction}
In \cite{dobrowolski_kim-independence_2022} we proved various properties for Kim-independence in thick positive NSOP$_1$ theories. One of these properties is the Independence Theorem \cite[Theorem 7.7]{dobrowolski_kim-independence_2022}. There is a gap in the proof there, however the statement remains true. In this corrigendum we present an alternative proof, based on the proof in full first-order logic from \cite{kaplan_kim-independence_2020}.

%This is a standalone version of the \note{published} corrigendum \note{[citation]}. The main contents are the same, but this version includes some extra sections (Sections \ref{sec:introduction}--\ref{sec:the-gap}): namely this introduction, the preliminary section that recalls the context from the original paper and a section with \note{more details} on the gap in the original proof.

\textbf{Overview.} After discussing some preliminaries in Section \ref{sec:preliminaries}, we discuss the gap in Section \ref{sec:the-gap}. We then develop some technical tools in Sections \ref{sec:technical-tools} and \ref{sec:spread-out-trees}, where the former focuses on general tools and the latter focuses on so-called spread out trees, an important concept for Kim-independence in NSOP$_1$ theories. Finally, in Section \ref{sec:independence-theorem} we present the new proof of the Independent Theorem.

\textbf{Acknowledgements.} We would like to thank W.\ Kho, who found the gap while working on his Master's thesis \cite{kho_primer_2023}. We would also like to thank I.\ Kaplan and N.\ Ramsey for discussions leading to this proof.

\section{Preliminaries}
\label{sec:preliminaries}
We recall some basic definitions and conventions, and refer to \cite[Sections 2 and 3]{dobrowolski_kim-independence_2022} for more details, explanations and references.
\begin{convention}
We establish some basic conventions.
\thlabel{basic-conventions}
\begin{enumerate}[label=(\roman*)]
\item We will drop ``positive'' everywhere, so ``formula'', ``type'' and ``theory'' always mean ``positive (existential) formula'', ``positive (existential) type'' and ``positive theory'' (or ``h-inductive theory'') respectively.
\item Whenever we say \emph{type} we mean ``maximal type'', otherwise we say \emph{partial type}.
\item We work in a monster model $\MM$, which we often drop from the notation.
\item We assume all parameter sets considered to be small (w.r.t.\ the monster), unless we consider the monster as a parameter set. Lowercase Latin letters $a, b, \ldots$ will denote (possibly infinite) tuples in the monster and uppercase Latin letters $A, B, \ldots$ will denote parameter sets in the monster. We will use $M$ and $N$ when these sets are e.c.\ models. Lowercase Greek letters $\alpha, \beta, \ldots$ are used for realisations of global (partial) types (in a bigger monster).
\end{enumerate}
\end{convention}
\begin{definition}
We recall some basic definitions.
\thlabel{basic-definitions}
\begin{enumerate}[label=(\roman*)]
\item The notation $\d_C(a, b) \leq n$ means that there are $a = a_0, a_1, \ldots, a_n = b$ such that $a_i$ and $a_{i+1}$ are on a $C$-indiscernible sequence for all $0 \leq i < n$.
\item A theory $T$ is called \emph{thick} if being an indiscernible sequence is type-definable. That is, there is a partial type $\Theta((x_i)_{i < \omega})$ such that $\models \Theta((a_i)_{i < \omega})$ iff $(a_i)_{i < \omega}$ is indiscernible. This is equivalent to having $\d_z(x, y) \leq n$ type-definable for all $n < \omega$.
\item For a theory $T$ we define the cardinal $\lambda_T = \beth_{(2^{|T|})^+}$.
\item A Lascar strong type $\Lstp(a/N)$ with $M \subseteq N$ is called \emph{$M$-Ls-invariant} if for any $b, b' \in N$ with $b \equivls_M b'$ we have $ab \equivls_M ab'$.
\item Suppose that $q(x)$ is a global $M$-Ls-invariant type $q(x)$. A \emph{Morley sequence in $q$ (over $M$)} is a sequence $(a_i)_{i < \omega}$ such that $(a_i)_{i < \omega} \equivls_M (\alpha_i)_{i < \omega}$ where $(\alpha_i)_{i < \omega} \models q^{\otimes \omega}$.
\end{enumerate}
\end{definition}
We recall some facts that will be used throughout, often implicitly.
\begin{fact}
\thlabel{useful-facts}
Let $T$ be a thick theory.
\begin{enumerate}[label=(\roman*)]
\item For any $\kappa \geq |A| + |T|$ there is $\kappa^+$-saturated $N \supseteq A$ with $|N| \leq 2^\kappa$.
\item Let $M$ be $\lambda_T$-saturated, then $a \equiv_M a'$ implies $\d_M(a, a') \leq 2$.
\item If $N \supseteq C$ is $(2^{|C| + \lambda_T})^+$-saturated and $q(x)$ and $r(x)$ are global $C$-Ls-invariant types with $q|_N = r|_N$ then $q = r$.
\item For any $a$ and $M$, $\Lstp(a/M)$ extends to some global $M$-Ls-invariant type.
\item If $q$ is a global $M$-Ls-invariant type and $B \supseteq M$ then $q$ is also $B$-Ls-invariant.
\end{enumerate}
\end{fact}
\begin{proof}
All of these facts can be found in \cite{dobrowolski_kim-independence_2022}, specifically:
\begin{enumerate}[label=(\roman*)]
\item Fact 2.12,
\item Lemma 2.20,
\item Fact 7.6,
\item Corollary 3.11.
\item Lemma 3.8(ii).
\end{enumerate}
\end{proof}

\section{The gap}
\label{sec:the-gap}
The proof of the Independence Theorem for Kim-independence in thick NSOP$_1$ theories \cite[Theorem 7.7]{dobrowolski_kim-independence_2022} contains a gap, which we outline in this section. The theorem, as stated, is still true. However, to the best of the authors' knowledge, the gap cannot easily be fixed, so we give a different proof.

We refer to the notation used in the proof of \cite[Theorem 7.7]{dobrowolski_kim-independence_2022}. Everything is fine up to the point where it is argued how the theorem follows from what is called ``Claim 2'' (at the bottom of Page 88). By compactness, an $M$-indiscernible sequence $(g_i h_i g'_i h'_i g''_i h''_i)_{i \in \Z}$ is extracted from the data from Claim 2. However, it may be that the properties $(h''_i g''_{i+1})_{i \in \Z} \models (q'|_{z_0, y})^{\otimes \Z}|_M$ and $h_i g_{i+1} \equiv_{M h_{>i} g_{>i+1} h''_{>i} g''_{i+1}} h''_i g''_{i+1}$ are not carried over.

In more detail, the intended argument was as follows. By Claim 2 and compactness we find some sequence $(a_i b_i a'_i b'_i a''_i b''_i)_{i < \lambda}$ with those properties, where $\lambda$ is sufficiently large. Then we would let $(g_i h_i g'_i h'_i g''_i h''_i)_{i \in \Z}$ be an $M$-indiscernible sequence based on $(a_i b_i a'_i b'_i a''_i b''_i)_{i < \lambda}$ using \cite[Lemma 2.17]{dobrowolski_kim-independence_2022}. However, the desired properties are only guaranteed for consecutive elements of $(a_i b_i a'_i b'_i a''_i b''_i)_{i < \lambda}$, whereas all that we get is that for any $i$ there are some $j < k < \lambda$ such that
\[
g_i h_i g'_i h'_i g''_i h''_i g_{i+1} h_{i+1} g'_{i+1} h'_{i+1} g''_{i+1} h''_{i+1} \equiv_M a_j b_j a'_j b'_j a''_j b''_j a_k b_k a'_k b'_k a''_k b''_k.
\]

Re-indexing the sequence (e.g.\ shifting part of the indices by one) does not work, as there are other conditions that we also need to be satisfied (i.e.\ $(g'_i h'_{i+1})_{i \in \Z} \models (q'|_{y_0, z})^{\otimes \Z}|_M$ and $g_i h_i \equiv_{M g_{>i} h_{>i} g'_{>i} h'_{>i}} g'_i h'_i$). We also really need an $M$-indiscernible sequence to apply \cite[Lemma 2.28]{dobrowolski_kim-independence_2022}. Any attempt (that the authors have tried) to make the sequences indiscernible in some other way, for example by trying to work with str-indiscernible trees, ultimately runs into similar issues.

\section{Technical tools}
\label{sec:technical-tools}
We reformulate the chain condition in a form that will be useful to us.
\begin{lemma}[chain condition]
\thlabel{chain-condition}
Let $T$ be a thick NSOP$_1$ theory. Suppose that $a \ind^K_M b$ and that $(b_i)_{i < \omega}$ is a Morley sequence in some global $M$-Ls-invariant type with $b_0 = b$. Then, writing $p(x, b) = \tp(a/Mb)$, we have that
\[
\bigcup_{i < \omega} p(x, b_i)
\]
does not Kim-divide over $M$.
\end{lemma}
\begin{proof}
Let $q(x)$ be the global $M$-Ls-invariant type in which $(b_i)_{i < \omega}$ is a Morley sequence. As $a \ind^K_M b$ we have by \cite[Proposition 4.2]{dobrowolski_kim-independence_2022} that there is $Ma$-indiscernible $(b'_i)_{i < \omega} \models q^{\otimes \omega}|_M$ with $b'_0 = b$. So we have $(b'_i)_{i < \omega} \equiv_M (b_i)_{i < \omega}$ and we let $a^*$ be such that $a (b'_i)_{i < \omega} \equiv_M a^* (b_i)_{i < \omega}$. Then $(b_i)_{i < \omega}$ is $Ma^*$-indiscernible, and so $a^* \ind^K_M (b_i)_{i < \omega}$ by \cite[Lemma 6.1]{dobrowolski_kim-independence_2022}. We conclude by noting that $a^* b_i \equiv_M a^* b_0 \equiv_M a b'_0 = ab$ for all $i < \omega$.
\end{proof}
\begin{proposition}[Being Ls-invariant is type-definable]
\thlabel{ls-invariance-type-definable}
Let $T$ be a thick theory. Let $C$ be some parameter set and let $N \supseteq C$ be $(2^{|C|+\lambda_T})^+$-saturated (possibly $N$ is the monster). Define $\Sigma(x)$ to be the following partial type over $N$
\[
\bigcup \{ \d_C(xb, xb') \leq 2 : b, b' \in N \text{ are finite tuples such that } \d_C(b, b') \leq 1 \}.
\]
Then a type $q(x)$ over $N$ is $C$-Ls-invariant iff $\Sigma(\alpha)$ for $\alpha \models q$.
\end{proposition}
\begin{proof}
Let $q(x)$ be a $C$-Ls-invariant type over $N$ and let $\alpha \models q$. Let $b, b' \in N$ be finite tuples such that $\d_C(b, b') \leq 1$. Then there is a $C$-indiscernible sequence $(b_i)_{i < \omega}$ with $b_0 b_1 = bb'$, which we may assume to be in $N$ by saturation. Using saturation again, we find a $\lambda_T$-saturated $C \subseteq M \subseteq N$ such that $(b_i)_{i < \omega}$ is $M$-indiscernible. In particular this means that $bM \equivls_C b'M$ and so $\alpha bM \equivls_C \alpha b' M$. It follows that $\alpha b \equivls_M \alpha b'$ and thus by our choice of $M$ we get $\d_C(\alpha b, \alpha b') \leq 2$. As $b, b'$ were arbitrary, we conclude that $\models \Sigma(\alpha)$.

For the other direction we let $q(x)$ be a type over $N$ such that for $\alpha \models q$ we have $\models \Sigma(\alpha)$. Now let $d, d' \in N$ be (potentially infinite tuples) such that $d \equivls_C d'$. Let $n < \omega$ be such that $\d_C(d, d') \leq n$, we claim that $\d_C(\alpha d, \alpha d') \leq 2n$, which implies the required $\alpha d \equivls_C \alpha d'$. By thickness we have that the partial type $\d_C(\alpha d, \alpha d') \leq 2n$ is given by
\[
\bigcup \{ \d_C(\alpha b, \alpha b') \leq 2n : b \subseteq d \text{ and } b' \subseteq d' \text{ are finite matching tuples}\}.
\]
So we have reduced the problem to the case where $d$ and $d'$ are finite. By saturation then there are $d = d_0, d_1, \ldots, d_n = d'$ in $N$ such that $\d_C(d_i, d_{i+1}) \leq 1$ for all $0 \leq i < n$. By assumption we thus have that $\d_C(\alpha d_i, \alpha d_{i+1}) \leq 2$ for all $0 \leq i < n$. We conclude that $\d_C(\alpha d, \alpha d') \leq 2n$, as required.
\end{proof}
\begin{proposition}[Extending Ls-invariant types]
\thlabel{extend-ls-invariant-type}
Let $T$ be a thick theory. Let $N \supseteq C$ be $(2^{|C| + \lambda_T})^+$-saturated. Suppose that $p(x) = \tp(a/N)$ is a $C$-Ls-invariant type, then $p(x)$ extends to a unique global $C$-Ls-invariant type $q(x)$.
\end{proposition}
\begin{proof}
Let $\Sigma(x)$ be the global partial type from \thref{ls-invariance-type-definable} expressing $C$-Ls-invariance. We will show that $p(x) \cup \Sigma(x)$ is finitely satisfiable. So let $\phi(x, e) \in p(x)$, where $e$ is a tuple of parameters from $N$, and let $\Sigma_0(x) \subseteq \Sigma(x)$ be finite. Let $b_1, \ldots, b_n$ and $b'_1, \ldots, b'_n$ be the finite tuples that occur in $\Sigma_0(x)$, so $\d_C(b_i, b'_i) \leq 1$ for all $1 \leq i \leq n$. By saturation of $N$ we find $d_1, \ldots, d_n, d'_1, \ldots, d'_n \in N$ such that $d_1 \ldots d_n d'_1 \ldots d'_n \equiv_{Ce} b_1 \ldots b_n b'_1 \ldots b'_n$. So for all $1 \leq i \leq n$ we have $\d_C(d_i, d'_i) \leq 1$, and hence $\d_C(a d_i, a d'_i) \leq 2$ by \thref{ls-invariance-type-definable} applied to $p(x)$. Now let $a^*$ be such that $a d_1 \ldots d_n d'_1 \ldots d'_n \equiv_{Ce} a^* b_1 \ldots b_n b'_1 \ldots b'_n$. Then by construction we have that $\models \phi(a^*, e)$ and $\models \Sigma_0(a^*)$, which proves finite satisfiability of $p(x) \cup \Sigma(x)$. By compactness we then find a realisation $\alpha$ of $p(x) \cup \Sigma(x)$, so that $q(x) = \tp(\alpha/\MM)$ is our desired $C$-Ls-invariant type. The uniqueness claim follows from \thref{useful-facts}(iii).
\end{proof}
We recall from \cite[Definition 3.12]{dobrowolski_kim-independence_2022} that $a \ind^{iLs}_C b$ means that $\tp(a/Cb)$ extends to a global $C$-Ls-invariant type.
\begin{proposition}
\thlabel{indiscernible-invariant-independent-sequence-is-morley}
Let $T$ be a thick theory. If $(a_i)_{i < \omega}$ is a $C$-indiscernible sequence such that $a_i \ind^{iLs}_C a_{<i}$ for all $i < \omega$ then $(a_i)_{i < \omega}$ is a Morley sequence in some global $C$-Ls-invariant type.
\end{proposition}
\begin{proof}
By compactness we find $a_\omega$ such that $(a_i)_{i \leq \omega}$ is $C$-indiscernible. Set $p(x) = \tp(a_\omega/C a_{< \omega})$ and let $\Sigma(x)$ be the global partial type from \thref{ls-invariance-type-definable}. We claim that $p(x) \cup \Sigma(x)$ is consistent. Indeed, for any finite $p'(x) \subseteq p(x)$ there is some $i < \omega$ so that $p'(x)$ only contains parameters in $C a_{<i}$, and so $\models p'(a_i)$ by $C$-indiscernibility. As $a_i \ind^{iLs}_C a_{<i}$ we then have that $p'(x)$ extends to a global $C$-Ls-invariant type $q'(x)$, and any realisation of $q'(x)$ will then be a realisation of $p'(x) \cup \Sigma(x)$. So $p(x) \cup \Sigma(x)$ is finitely satisfiable and hence consistent.

Let $\alpha^*$ be a realisation of $p(x) \cup \Sigma(x)$ and set $q^*(x) = \tp(\alpha^*/\MM)$, so $q^*(x)$ is global $C$-Ls-invariant. Let $a^* \equivls_{C a_{< \omega}} \alpha^*$, then there is $f \in \Aut(\MM/C a_{< \omega})$ such that $f(a^*) = a_\omega$. Set $q = f(q^*)$, so $q(x)$ is global $C$-Ls-invariant by \cite[Lemma 3.8(i)]{dobrowolski_kim-independence_2022} with $p(x) \subseteq q(x)$ and for any $\alpha \models q$ we have $\alpha \equivls_{C a_{< \omega}} a_\omega$.

For any $i < \omega$ we thus have $a_i \equivls_{C a_{<i}} a_\omega \equivls_{C a_{<i}} \alpha$. We therefore have $a_i \models q^{\otimes i}|_C$ for all $i < \omega$ and so $(a_i)_{i < \omega} \models q^{\otimes \omega}|_C$. So $(a_i)_{i < \omega}$ is the automorphic image over $C$ of a Morley sequence over $C$, hence it is itself a Morley sequence in a (potentially different) global $C$-Ls-invariant type.
\end{proof}

\section{Spread out trees}
\label{sec:spread-out-trees}
We recall various definitions concerning trees and trees of parameters (which we will from now on also simply call trees) from \cite{kaplan_kim-independence_2020}. In particular, we will work with the ill-founded trees $\T_\alpha$ from \cite[Definition 5.1]{kaplan_kim-independence_2020} and we use the same notation, so we assume familiarity with those definitions. We refer to \cite{kamsma_positive_2023} for the definitions and terminology involving s-indiscernibility, str-indiscernibility and generalised $\EM$-types. We slightly adjust \cite[Definition 5.7]{kaplan_kim-independence_2020} to fit our situation.
\begin{definition}
\thlabel{spread-out-tree}
Let $(a_\eta)_{\eta \in \T_\alpha}$ be a tree and let $M$ be an e.c.\ model.
\begin{enumerate}[label=(\roman*)]
\item We call $(a_\eta)_{\eta \in \T_\alpha}$ \emph{spread out over $M$} if for all $\eta \in \T_\alpha$ with $\dom(\eta) = [\beta + 1, \alpha)$ for some $\beta < \alpha$, there is a global $M$-Ls-invariant type $q_\eta \supseteq \tp(a_{\unrhd \eta^\frown \langle 0 \rangle}/M)$ such that $(a_{\unrhd \eta^\frown \langle i \rangle})_{i < \omega}$ is a Morley sequence in $q_\eta$ over $M$.
\item A \emph{Morley tree over $M$} is an str-indiscernible and spread out tree over $M$.
\item A \emph{tree Morley sequence over $M$} is a branch in an infinite height Morley tree over $M$.
\end{enumerate}
\end{definition}
\begin{lemma}
\thlabel{sub-tree-morley-sequences}
Suppose that $(a_i)_{i < \omega}$ is a tree Morley sequence over $M$.
\begin{enumerate}[label=(\roman*)]
\item If $b_i \subseteq a_i$ for each $i < \omega$, of matching length and position, then $(b_i)_{i < \omega}$ is a tree Morley sequence over $M$.
\item Fix $1 \leq n < \omega$ and define $d_i = (a_{n i}, \ldots, a_{n i + n - 1})$ for all $i < \omega$. Then $(d_i)_{i < \omega}$ is a tree Morley sequence over $M$.
\end{enumerate}
\end{lemma}
\begin{proof}
This is essentially \cite[Lemma 5.9]{kaplan_kim-independence_2020}, but we work with slightly different definitions, so we go through the proof here. Part (i) is clear, because being a Morley tree is preserved under taking subtuples. For (ii) we let $(b_\eta)_{\eta \in \T_\omega}$ be a Morley tree such that $(a_i)_{i < \omega}$ is a branch in $(b_\eta)_{\eta \in \T_\omega}$. We may assume that $(a_i)_{i < \omega}$ is the branch indexed by the constant zero functions. We define $j: \T_\omega \to \T_\omega$ so that for $\eta \in \T_\omega$ with $\dom(\eta) = [k, \omega)$ we have $\dom(j(\eta)) = [nk+n-1, \omega)$ and
\[
j(\eta)(m) = \begin{cases}
\eta((m-(n-1))/n) & \text{if } n | (m-(n-1)), \\
0 & \text{otherwise},
\end{cases}
\]
for all $m \in [nk + n - 1, \omega)$. We define $(c_\eta)_{\eta \in \T_\omega}$ by $c_\eta = (b_{j(\eta)}, \ldots, b_{j(\eta)^\frown \langle 0 \rangle^{n-1}})$. This corresponds to the $n$-fold elongation of $(b_\eta)_{\eta \in \T_\omega}$ from \cite{chernikov_model-theoretic_2016}. One then straightforwardly verifies that $(c_\eta)_{\eta \in \T_\omega}$ is a Morley tree over $M$, so $(c_{\zeta_i})_{i < \omega}$ is a tree Morley sequence over $M$. For $i < \omega$ we have
\[
c_{\zeta_i} = (b_{\zeta_{ni+n-1}}, \ldots, b_{\zeta_{ni}}) = (a_{ni+n-1}, \ldots, a_{ni}),
\]
so by reversing the order of the tuples we see that $(d_i)_{i < \omega}$ is a tree Morley sequence over $M$.
\end{proof}
\begin{lemma}[Kim's lemma for tree Morley sequences]
\thlabel{kims-lemma-tree-morley-sequences}
Let $T$ be a thick NSOP$_1$ theory. Let $M$ be an e.c.\ model and let $\Sigma(x, b)$ be a partial type over $M$. Then the following are equivalent:
\begin{enumerate}[label=(\roman*)]
\item $\Sigma(x, b)$ Kim-divides over $M$;
\item for some tree Morley sequence $(b_i)_{i < \omega}$ over $M$ with $b_0 = b$ we have that $\bigcup_{i < \omega} \Sigma(x, b_i)$ is inconsistent;
\item for every tree Morley sequence $(b_i)_{i < \omega}$ over $M$ with $b_0 = b$ we have that $\bigcup_{i < \omega} \Sigma(x, b_i)$ is inconsistent.
\end{enumerate} 
\end{lemma}
\begin{proof}
This is \cite[Corollary 5.14]{kaplan_kim-independence_2020}, whose proof is really found in \cite[Proposition 5.13]{kaplan_kim-independence_2020}. Our setting requires some minor extra verifications, which we will do below, but the proof is essentially the same.

Given the existence of tree Morley sequences starting with $b$ (\thref{tree-morley-sequences-exist}), the equivalence of these three statements reduces to proving that for any tree Morley sequence $(b_i)_{i < \omega}$ over $M$ with $b_0 = b$ we have that $\Sigma(x, b)$ Kim-divides iff $\bigcup_{i < \omega} \Sigma(x, b_i)$ is inconsistent.

Let $(c_\eta)_{\eta \in \T_\omega}$ be a Morley tree over $M$ such that $(b_i)_{i < \omega}$ is a branch in that tree, which we may assume to be the constant zero branch. For $i < \omega$ define $\eta_i \in \T_\omega$ to be the function with domain $[i, \omega)$ such that
\[
\eta_i(j) = \begin{cases}
1 & \text{if } i = j,\\
0 & \text{otherwise}.
\end{cases}
\]
By str-indiscernibility, the sequences $(c_{\zeta_i})_{i < \omega}$ and $(c_{\eta_i})_{i < \omega}$ are $M$-indiscernible. We claim that $(c_{\eta_i})_{i < \omega}$ is a Morley sequence over $M$ in a global $M$-Ls-invariant type. Indeed, because $(c_\eta)_{\eta \in \T_\omega}$ is spread out over $M$ we have that $c_{\eta_i} \ind^{iLs}_M (c_{\eta_j})_{j < i}$ for all $i < \omega$. So the claim follows from \thref{indiscernible-invariant-independent-sequence-is-morley}. By str-indiscernibility we also have for all $i < \omega$ that $c_{\zeta_i}, c_{\eta_i}$ starts an $M (c_{\zeta_j}, c_{\eta_j})_{j > i}$-indiscernible sequence. So since $T$ is NSOP$_1$ we can apply \cite[Lemma 5.10]{dobrowolski_kim-independence_2022} to conclude that $\bigcup_{i < \omega} \Sigma(x, c_{\zeta_i})$ is inconsistent iff $\bigcup_{i < \omega} \Sigma(x, c_{\eta_i})$ is inconsistent. The former is just $\bigcup_{i < \omega} \Sigma(x, b_i)$, and the latter is inconsistent iff $\Sigma(x, b)$ Kim-divides by Kim's lemma for NSOP$_1$ theories \cite[Proposition 4.4]{dobrowolski_kim-independence_2022}, which concludes the proof.
\end{proof}
\begin{fact}[Tree modelling theorems]
\thlabel{tree-modelling}
Let $T$ be a thick theory.
\begin{enumerate}[label=(\roman*)]
\item Let $(a_\eta)_{\eta \in \T_\alpha}$ be a tree of tuples and let $C$ be any set of parameters, then there is a tree $(b_\eta)_{\eta \in \T_\alpha}$ that is s-indiscernible over $C$ and $\EM_s$-based on $(a_\eta)_{\eta \in \T_\alpha}$ over $C$.
\item Let $C$ be any parameter set, $\kappa$ any cardinal, and let $\lambda = \beth_{(2^{|T| + |C| + \kappa})^+}$. Given any tree $(a_\eta)_{\eta \in \T_\lambda}$ of $\kappa$-tuples that is s-indiscernible over $C$, there is a tree $(b_\eta)_{\eta \in \T_\omega}$ that is str-indiscernible over $C$ str-based on $(a_\eta)_{\eta \in \T_\lambda}$ over $C$. The latter means that for any finite tuple $\bar{\eta} \in \T_\omega$ there is $\bar{\nu} \in \T_\lambda$ such that $\bar{\eta}$ and $\bar{\nu}$ have the same str-quantifier-free type and $b_{\bar{\eta}} \equiv_C a_{\bar{\nu}}$.
\end{enumerate}
\end{fact}
\begin{proof}
Part (i) is \cite[Theorem 4.6]{kamsma_positive_2023}, which is essentially just compactness applied to \cite[Proposition 5.8]{dobrowolski_kim-independence_2022}. Part (ii) is \cite[Theorem 4.8]{kamsma_positive_2023}, which is technically stated for well-founded trees, but its proof applies to the ill-founded trees we are interested in here.
\end{proof}
\begin{lemma}
\thlabel{spread-out-preserved}
Let $T$ be a thick theory. Suppose that $(a_\eta)_{\eta \in \T_\alpha}$ is s-indiscernible and spread out over $M$ and that $(b_\eta)_{\eta \in \T_\omega}$ is str-based on $(a_\eta)_{\eta \in \T_\alpha}$ over $M$, then $(b_\eta)_{\eta \in \T_\omega}$ is spread out over $M$.
\end{lemma}
\begin{proof}
Let $\eta \in \T_\omega$, we have to show that $(b_{\unrhd \eta^\frown \langle i \rangle})_{i < \omega}$ is a Morley sequence in some global $M$-Ls-invariant type. We claim that $b_{\unrhd \eta^\frown \langle i \rangle} \ind^{iLs}_M (b_{\unrhd \eta^\frown \langle j \rangle})_{j < i}$ for all $i < \omega$. This is indeed enough, because $(b_{\unrhd \eta^\frown \langle i \rangle})_{i < \omega}$ is $M$-indiscernible by str-indiscernibility over $M$, and so the result follows by \thref{indiscernible-invariant-independent-sequence-is-morley}.

We prove the claim by showing that for all $i < \omega$ and all finite $b \subseteq b_{\unrhd \eta^\frown \langle i \rangle}$ and $b' \subseteq (b_{\unrhd \eta^\frown \langle j \rangle})_{j < i}$ we have $b \ind^{iLs}_M b'$, which is enough by \thref{ls-invariance-type-definable}. Let $\bar{\nu}_{i_1}, \ldots, \bar{\nu}_{i_n}$ be finite tuples in $\T_\omega$ such that $i_1 < \ldots < i_n < \omega$ and $\bigwedge \bar{\nu}_{i_k} \unrhd \eta^\frown \langle i_k \rangle$ for all $1 \leq k \leq n$. By str-basing there are $\gamma, \bar{\mu}_{i_1}, \ldots, \bar{\mu}_{i_n}$ in $\T_\alpha$ such that $\gamma \bar{\mu}_{i_1} \ldots \bar{\mu}_{i_n}$ has the same str-quantifier-free type as $\eta \bar{\nu}_{i_1} \ldots \bar{\nu}_{i_n}$ and $b_{\eta} b_{\bar{\nu}_{i_1}} \ldots b_{\bar{\nu}_{i_n}} \equiv_M a_\gamma a_{\bar{\mu}_{i_1}} \ldots a_{\bar{\mu}_{i_n}}$. We now have reduced the problem to showing that $a_{\bar{\mu}_{i_n}} \ind^{iLs}_M a_{\bar{\mu}_{i_1}} \ldots a_{\bar{\mu}_{i_{n-1}}}$. As $\gamma \lhd \bigwedge \bar{\mu}_{i_n}$, there must be some $m < \omega$ such that $\bigwedge \bar{\mu}_{i_n} \unrhd \gamma^\frown \langle m \rangle$. Furthermore, we have for every $1 \leq k < n$ that $\gamma \lhd \bigwedge \bar{\mu}_{i_k}$ and $\bigwedge \bar{\mu}_{i_k} <_\lex \bigwedge \bar{\mu}_{i_n}$, and so $\bigwedge \bar{\mu}_{i_k} \unrhd \gamma^\frown \langle j \rangle$ for some $j < m$. Because $(a_\eta)_{\eta \in \T_\alpha}$ is spread out over $M$ we have $a_{\unrhd \gamma^\frown \langle m \rangle} \ind^{iLs}_M (a_{\unrhd \gamma^\frown \langle j \rangle})_{j < m}$, and so $a_{\bar{\mu}_{i_n}} \ind^{iLs}_M a_{\bar{\mu}_{i_1}} \ldots a_{\bar{\mu}_{i_{n-1}}}$, as required.
\end{proof}
\begin{corollary}
\thlabel{base-morley-tree-on-s-indiscernible-tree}
Let $T$ be a thick theory, and let $C$ be some parameter set and $\kappa$ some cardinal. Set $\lambda = \beth_{\left( 2^{2^{\lambda_T + |C|} + \kappa} \right)^+}$. Given a tree $(a_\eta)_{\eta \in \T_\lambda}$ that is s-indiscernible and spread out over $C$, there is a Morley tree $(b_\eta)_{\eta \in \T_\omega}$ over $C$ that is str-Ls-based on $(a_\eta)_{\eta \in \T_\lambda}$ over $C$. The latter means that for any finite tuple $\bar{\eta} \in \T_\omega$ there is $\bar{\nu} \in \T_\lambda$ such that $\bar{\eta}$ and $\bar{\nu}$ have the same str-quantifier-free type and $b_{\bar{\eta}} \equivls_C a_{\bar{\nu}}$.
\end{corollary}
\begin{proof}
By \thref{useful-facts}(i) there is $\lambda_T$-saturated $M \supseteq C$ with $|M| \leq 2^{\lambda_T + |C|}$. Using \thref{tree-modelling}(ii) we find a tree $(b_\eta)_{\eta \in \T_\omega}$ that is str-indiscernible over $M$ and str-based on $(a_\eta)_{\eta \in \T_\lambda}$ over $M$. In particular $(b_\eta)_{\eta \in \T_\omega}$ is str-based on $(a_\eta)_{\eta \in \T_\lambda}$ over $C$, so it is spread out over $C$ by \thref{spread-out-preserved} and hence it is a Morley tree over $C$. Finally, by str-basing, we have that for any finite tuple $\bar{\eta} \in \T_\omega$ there is $\bar{\nu} \in \T_\lambda$ such that $\bar{\eta}$ and $\bar{\nu}$ have the same str-quantifier-free type and $b_{\bar{\eta}} \equiv_M a_{\bar{\nu}}$. By our choice of $M$ this implies $b_{\bar{\eta}} \equivls_C a_{\bar{\nu}}$, as required.
\end{proof}
The following key lemma in constructing spread out trees is due to N.\ Ramsey, for which we take terminology from \cite[Definition 1.14]{chernikov_transitivity_2023}.
\begin{definition}
\thlabel{mutually-s-indiscernible}
We call a sequence of trees $((a^i_\eta)_{\eta \in \T_\alpha})_{i < \omega}$ \emph{mutually s-indiscernible over $C$} if $(a^i_\eta)_{\eta \in \T_\alpha}$ is s-indiscernible over $C((a^j_\eta)_{\eta \in \T_\alpha})_{j \neq i, j < \omega}$ for all $i < \omega$.
\end{definition}
\begin{lemma}
\thlabel{ramsey-fix}
Let $T$ be a thick theory and let $(a_\eta)_{\eta \in \T_\alpha}$ be a tree that is s-indiscernible over $M$. Then there is a Morley sequence $((a^i_\eta)_{\eta \in \T_\alpha})_{i < \omega}$ in some global $M$-Ls-invariant type with $(a^0_\eta)_{\eta \in \T_\alpha} = (a_\eta)_{\eta \in \T_\alpha}$ that is mutually s-indiscernible over $M$.
\end{lemma}
\begin{proof}
Let $q((x_\eta)_{\eta \in \T_\alpha}) \supseteq \tp((a_\eta)_{\eta \in \T_\alpha}/M)$ be a global $M$-Ls-invariant type. Let $N \supseteq M$ be $(2^{|M| + \lambda_T})^+$-saturated, and let $(a'_\eta)_{\eta \in \T_\alpha} \models q|_N$. Apply the s-modelling theorem (\thref{tree-modelling}(i)) to find a tree $(a''_\eta)_{\eta \in \T_\alpha}$ that is s-indiscernible over $N$ and $\EM_s$-based on $(a'_\eta)_{\eta \in \T_\alpha}$ over $N$.
\begin{claim}
\thlabel{ramsey-fix:m-ls-invariance}
The type $\tp((a''_\eta)_{\eta \in \T_\alpha}/N)$ is $M$-Ls-invariant.
\end{claim}
\begin{proof}[Proof of claim]
By \thref{ls-invariance-type-definable} it is enough to show that for any finite $b,b' \in N$ with $\d_M(b, b') \leq 1$ we have $\d_M((x_\eta)_{\eta \in \T_\alpha} b, (x_\eta)_{\eta \in \T_\alpha} b') \leq 2 \subseteq \tp((a''_\eta)_{\eta \in \T_\alpha}/N)$. By thickness we have that $\d_M((x_\eta)_{\eta \in \T_\alpha} b, (x_\eta)_{\eta \in \T_\alpha} b') \leq 2$ is given by
\[
\bigcup \{ \d_M(x_{\bar{\eta}} b, x_{\bar{\eta}} b') \leq 2 : \bar{\eta} \text{ is a finite tuple in } \T_\alpha \}.
\]
Let $\bar{\eta}$ be any finite tuple in $\T_\alpha$. For any $\bar{\nu}$ that has the same s-quantifier-free type as $\bar{\eta}$ we have that $\d_M(x_{\bar{\nu}} b, x_{\bar{\nu}} b') \leq 2 \subseteq \tp((a'_\eta)_{\eta \in \T_\alpha}/N)$ by \thref{ls-invariance-type-definable}, because $\tp((a'_\eta)_{\eta \in \T_\alpha}/N) = q|_N$ is $M$-Ls-invariant. We thus see that $\d_M(x_{\bar{\eta}} b, x_{\bar{\eta}} b') \leq 2 \subseteq \EM_s((a'_\eta)_{\eta \in \T_\alpha}/N) \subseteq \tp((a''_\eta)_{\eta \in \T_\alpha}/N)$, which concludes the proof of the claim.
\end{proof}
By \thref{ramsey-fix:m-ls-invariance}, \thref{extend-ls-invariant-type} and our choice of $N$ there is a unique global $M$-Ls-invariant type $q''((x_\eta)_{\eta \in \T_\alpha}) \supseteq \tp((a''_\eta)_{\eta \in \T_\alpha}/N)$. Let $((b^i_\eta)_{\eta \in \T_\alpha})_{i < \omega}$ be a Morley sequence in $q''$ over $N$.
\begin{claim}
\thlabel{ramsey-fix:s-indiscernibility}
The sequence $(b^i_\eta)_{\eta \in \T_\alpha}$ is mutually s-indiscernible over $N$.
\end{claim}
\begin{proof}[Proof of claim]
Fix $i < \omega$. We prove by induction on $k \geq i$ that $(b^i_\eta)_{\eta \in \T_\alpha}$ is s-indiscernible over $N((b^j_\eta)_{\eta \in \T_\alpha})_{j \neq i, j < k}$.

For the base case $k = i$ we need to prove that  $(b^i_\eta)_{\eta \in \T_\alpha}$ is s-indiscernible over $N((b^j_\eta)_{\eta \in \T_\alpha})_{j < i}$. Let $\bar{\eta}, \bar{\nu} \in \T_\alpha$ be finite tuples with the same s-quantifier-free type. As $(b^i_\eta)_{\eta \in \T_\alpha} \equiv_N (a''_\eta)_{\eta \in \T_\alpha}$, we have that it is s-indiscernible over $N$. So there is a single type (after renaming variables) $p(y) = \tp(b^i_{\bar{\eta}}/N) = \tp(b^i_{\bar{\nu}}/N)$, which is $M$-Ls-invariant by \thref{ramsey-fix:m-ls-invariance}. Since $q''(x_{\bar{\eta}})$ and $q''(x_{\bar{\nu}})$ are both global $M$-Ls-invariant extensions of $p(y)$ we have that $q''(x_{\bar{\eta}}) = q''(x_{\bar{\nu}})$, after renaming variables. By construction $b^i_{\bar{\eta}} \models q''(x_{\bar{\eta}})|_{N ((b^j_\eta)_{\eta \in \T_\alpha})_{j < i}}$ and $b^i_{\bar{\nu}} \models q''(x_{\bar{\nu}})|_{N ((b^j_\eta)_{\eta \in \T_\alpha})_{j < i}}$, so $b^i_{\bar{\eta}} \equiv_{N ((b^j_\eta)_{\eta \in \T_\alpha})_{j < i}} b^i_{\bar{\nu}}$ follows, as required.

For the successor step we have $k > i$, and we assume that $(b^i_\eta)_{\eta \in \T_\alpha}$ is s-indiscernible over $N((b^j_\eta)_{\eta \in \T_\alpha})_{j \neq i, j < k}$. Let $\bar{\eta}, \bar{\nu} \in \T_\alpha$ be finite tuples with the same s-quantifier-free type. By the induction hypothesis we have
\[
b^i_{\bar{\eta}} \equivls_{N ((b^j_\eta)_{\eta \in \T_\alpha})_{j \neq i, j < k}} b^i_{\bar{\nu}},
\]
where we get equivalence of Lascar-strong types instead of just normal types from s-indiscernibility (see e.g.\ \cite[Proposition 4.5]{kamsma_positive_2023}). As $(b^k_\eta)_{\eta \in \T_\alpha}$ realises an $M$-Ls-invariant type over $N ((b^j_\eta)_{\eta \in \T_\alpha})_{j < k}$ and $N \supseteq M$ we get
\[
(b^k_\eta)_{\eta \in \T_\alpha} b^i_{\bar{\eta}} \equivls_{N ((b^j_\eta)_{\eta \in \T_\alpha})_{j \neq i, j < k}} (b^k_\eta)_{\eta \in \T_\alpha} b^i_{\bar{\nu}},
\]
which completes the induction step and thus the proof of the claim.
\end{proof}
We have $(b^0_\eta)_{\eta \in \T_\alpha} \equiv_M (a''_\eta)_{\eta \in \T_\alpha} \equiv_M (a'_\eta)_{\eta \in \T_\alpha} \equiv_M (a_\eta)_{\eta \in \T_\alpha}$, where the middle equality of types follows because $(a'_\eta)_{\eta \in \T_\alpha}$ is s-indiscernible over $M$ and so its $\EM_s$-type over $M$ is maximal (i.e.\ is the same as its type over $M$) and $(a''_\eta)_{\eta \in \T_\alpha}$ is in particular $\EM_s$-based on $(a'_\eta)_{\eta \in \T_\alpha}$ over $M$. So by an automorphism we find $((a^i_\eta)_{\eta \in \T_\alpha})_{i < \omega} \equiv_M ((b^i_\eta)_{\eta \in \T_\alpha})_{i < \omega}$, with $(a^0_\eta)_{\eta \in \T_\alpha} = (a_\eta)_{\eta \in \T_\alpha}$, which is then as required by construction of $((b^i_\eta)_{\eta \in \T_\alpha})_{i < \omega}$ and \thref{ramsey-fix:s-indiscernibility}.
\end{proof}
\begin{remark}
\thlabel{rem:missing-ingredient-kaplan-ramsey}
\thref{ramsey-fix} is in fact a missing ingredient in \cite{kaplan_kim-independence_2020}, in particular in the inductive steps in their Lemmas 5.11 and 6.4. There they replace some spread out tree $A$ by an s-indiscernible tree $B$ locally based on $A$ (in our terminology: $\EM_s$-based). However, this process might not preserve the property of being spread out. By replacing the inductive step by \thref{ramsey-fix}, the argument can be fixed.

In existing work on Kim-independence over arbitrary sets there is the same issue, as discussed in \cite[Page 7]{chernikov_transitivity_2023}. This can be fixed in a similar manner: \cite[Lemma 1.15]{chernikov_transitivity_2023} is a variant of \thref{ramsey-fix} over arbitrary sets (in full first-order logic), and can then be used in the inductive steps in the same way.

We also remark that this is not an issue in \cite{dobrowolski_kim-independence_2022}, because the proofs there make use of a different notion called ``$q$-spread-out''. The point of this notion is that it is type-definable, so it can be captured by the $\EM_s$-type. The gap in the proof of the Independence Theorem that this corrigendum addresses is of a different nature.
\end{remark}
The following lemma illustrates the use of \thref{ramsey-fix} and completes the proof of \thref{kims-lemma-tree-morley-sequences}.
\begin{lemma}
\thlabel{tree-morley-sequences-exist}
Let $T$ be a thick theory. For any $a$ and $M$ there is a tree Morley sequence $(a_i)_{i < \omega}$ over $M$ with $a_0 = a$.
\end{lemma}
\begin{proof}
Let $\lambda$ be the cardinal from \thref{base-morley-tree-on-s-indiscernible-tree}, where $M$ and $|a|$ take the respective roles of $C$ and $\kappa$ there. By induction on $\alpha \leq \lambda$ we will construct trees $(a^\alpha_\eta)_{\eta \in \T_\alpha}$, such that:
\begin{enumerate}
\item for all $\eta \in \T_\alpha$ we have $a^\alpha_\eta \equiv_M a$,
\item the tree $(a^\alpha_\eta)_{\eta \in \T_\alpha}$ is spread out and s-indiscernible over $M$,
\item for all $\beta < \alpha$ we have $a^\alpha_{\iota_{\beta \alpha}(\eta)} = a^\beta_\eta$ for all $\eta \in \T_\beta$.
\end{enumerate}
We start by setting $a^0_\emptyset = a$. For a limit stage $\ell$, we set $a^\ell_{\iota_{\beta \ell}(\eta)} = a^\beta_\eta$, where $\beta$ ranges over all ordinals $< \ell$ and $\eta$ ranges over all elements in $\T_\beta$. This is well-defined by property (3), and properties (1) and (2) follow immediately from the induction hypothesis.

For the successor step we suppose $(a^\alpha_\eta)_{\eta \in \T_\alpha}$ has been constructed. By \thref{ramsey-fix} we find a Morley sequence $((a^\alpha_{\eta,i})_{\eta \in \T_\alpha})_{i < \omega}$ in some global $M$-Ls-invariant type with $(a^\alpha_{\eta,0})_{\eta \in \T_\alpha} = (a^\alpha_\eta)_{\eta \in \T_\alpha}$ that is mutually s-indiscernible over $M$. Define a tree $(b_\eta)_{\eta \in \T_{\alpha+1}}$ by setting $b_\emptyset = a$ and $b_{\langle i \rangle^\frown \eta} = a^\alpha_{\eta,i}$ for all $\eta \in \T_\alpha$ and $i < \omega$. The $\EM_s$-type of $(b_\eta)_{\eta \in \T_{\alpha+1}}$ over $M$ satisfies the following properties.
\begin{enumerate}[label=(\roman*)]
\item It contains $\tp((b_\eta)_{\eta \in \T_{\alpha+1} \setminus \{\emptyset\}}/M)$. This is because $(b_{\rhd \langle i \rangle})_{i < \omega}$ forms an $M$-indiscernible sequence, and by the s-indiscernibility of the individual trees in this sequence.
\item The $\EM_s$-type specifies that the type of the root is $\tp(a/M)$.
\end{enumerate}
We apply \thref{tree-modelling}(i) to find an s-indiscernible tree $(a^{\alpha+1}_\eta)_{\eta \in \T_{\alpha+1}}$ over $M$ that is $\EM_s$-based over $M$ on $(b_\eta)_{\eta \in \T_{\alpha+1}}$. By an automorphism and (i) we may assume that $a^{\alpha+1}_{\langle i \rangle^\frown \eta} = b_{\langle i \rangle^\frown \eta} = a^\alpha_{\eta,i}$ for all $\eta \in \T_\alpha$ and $i < \omega$, and so (3) is satisfied. This then also implies that (2) is satisfied and (1) is satisfied by (ii), completing the inductive construction.

We thus have constructed a tree $(a^\lambda_\eta)_{\eta \in \T_\lambda}$ that is spread out and s-indiscernible over $M$ with $a^\lambda_\eta \equiv_M a$ for all $\eta \in \T_\lambda$. We can now apply \thref{base-morley-tree-on-s-indiscernible-tree} to find a Morley tree $(a_\eta)_{\eta \in \T_\omega}$ that is str-Ls-based on $(a^\lambda_\eta)_{\eta \in \T_\lambda}$ over $M$. In particular $a_\eta \equiv_M a$ for all $\eta \in \T_\omega$, and so by an automorphism we may assume $a_{\zeta_0} = a$. Then setting $a_i = a_{\zeta_i}$ for all $i < \omega$ we obtain the required tree Morley sequence $(a_i)_{i < \omega}$.
\end{proof}

\section{The Independence Theorem}
\label{sec:independence-theorem}
We now give a new proof of the Independence Theorem \cite[Theorem 7.7]{dobrowolski_kim-independence_2022}. The statement remains exactly the same. The proof is essentially that of \cite[Theorem 6.5]{kaplan_kim-independence_2020}, with \thref{ramsey-fix} mixed in.

\begin{definition}[{\cite[Definition 7.3]{dobrowolski_kim-independence_2022}}]
We write $a \ind^*_M b$ to mean that $\Lstp(a/Mb)$ extends to a global $M$-Ls-invariant type $\tp(N/\MM)$ for some $\beth_\omega(\lambda_T + |Mab|)$-saturated $N \supseteq M$. Extending $\Lstp(a/Mb)$ here means that there is some $\alpha \in N$ with $\alpha \equivls_{Mb} a$.
\end{definition}
We note that the relation $\ind^*$ is really invariant under automorphisms (and not just under Lascar strong automorphisms). Suppose that $q(x)$ is a global $M$-Ls-invariant type witnessing $a \ind^*_M b$ and $f$ is some automorphism of the monster, then $f(q)$ witnesses $f(a) \ind^*_{f(M)} f(b)$.
\begin{fact}[Weak Independence Theorem, {\cite[Proposition 7.5]{dobrowolski_kim-independence_2022}}]
\thlabel{weak-independence-theorem}
Let $T$ be a thick NSOP$_1$ theory. Suppose that $a \equivls_M a'$, $a \ind^K_M b$, $a' \ind^K_M c$ and $b \ind^*_M c$. Then there is $a''$ with $a'' \equivls_{Mb} a$, $a'' \equivls_{Mc} a'$ and $a'' \ind^K_M bc$.
\end{fact}
\begin{lemma}
\thlabel{tree-generalised-weak-independence-theorem}
Let $T$ be a thick NSOP$_1$ theory. Suppose that $a \ind^K_M b$ and fix some cardinal $\kappa$. Suppose that $q(x, y) = \tp(N/\MM)$ is a global $M$-Ls-invariant type extending $\Lstp(b/M)$, where $N \supseteq M$ is $\beth_\omega(\lambda_T + |Mab| + |\T_\kappa|)$-saturated and the $x$ variable matches $b$. If $(b_\eta)_{\eta \in \T_\alpha}$, with $\alpha \leq \kappa$, is a tree that is spread out over $M$, such that for all $\eta \in \T_\alpha$ we have $b_\eta \equivls_M b$ and $b_\eta \models q(x)|_{M b_{\rhd \eta}}$, then, writing $p(x, b) = \tp(a/Mb)$,
\[
\bigcup_{\eta \in \T_\alpha} p(x, b_\eta)
\]
does not Kim-divide over $M$.
\end{lemma}
\begin{proof}
We follow the proof of \cite[Lemma 6.2]{kaplan_kim-independence_2020}, replacing their use of \cite[Proposition 6.1]{kaplan_kim-independence_2020} by \thref{weak-independence-theorem}. The proof is by induction on $\alpha$. For $\alpha = 0$ there is nothing to do, and limit stages follow from the induction hypothesis by finite character. Now suppose that $(b_\eta)_{\eta \in \T_{\alpha + 1}}$ is as in the statement. By the induction hypothesis we have that
\[
\bigcup_{\eta \unrhd \langle 0 \rangle} p(x, b_\eta)
\]
does not Kim-divide over $M$. Because $(b_\eta)_{\eta \in \T_{\alpha + 1}}$ is spread out we have that $(b_{\unrhd \langle i \rangle})_{i < \omega}$ is a Morley sequence in some global $M$-Ls-invariant type. By the chain condition \thref{chain-condition} we then have that
\[
\bigcup_{i < \omega} \bigcup_{\eta \unrhd \langle i \rangle} p(x, b_\eta)
\]
does not Kim-divide over $M$. At the same time we have $b_\emptyset \models q(x)|_{M b_{\rhd \emptyset}}$ and so by our assumptions on $q$ we have $b_\emptyset \ind^*_M  b_{\rhd \emptyset}$. Using that $p(x, b_\emptyset)$ does not Kim-divide (because $a \ind^K_M b$), we can apply the Weak Independence Theorem (\thref{weak-independence-theorem}) to see that
\[
p(x, b_\emptyset) \cup \bigcup_{i < \omega} \bigcup_{\eta \unrhd \langle i \rangle} p(x, b_\eta)
\]
does not Kim-divide (here we implicitly used the assumption that $b_\eta \equivls_M b$ for all $\eta \in \T_{\alpha+1})$. Unfolding definitions, this is exactly saying that
\[
\bigcup_{\eta \in \T_{\alpha + 1}} p(x, b_\eta),
\]
does not Kim-divide, completing the induction step and thereby the proof.
\end{proof}
\begin{lemma}[Zig-Zag lemma]
\thlabel{zig-zag}
Let $T$ be a thick NSOP$_1$ theory. Suppose that $b \ind^K_M c$. Then there is a global $M$-Ls-invariant type $q(x,y) = \tp(N/\MM)$, where $N \supseteq M$ is some $\beth_\omega(\lambda_T + |Mbc|)$-saturated model and $q(x)$ extends $\tp(b/M)$, and a tree Morley sequence $(b_i, c_i)_{i < \omega}$ over $M$ such that:
\begin{enumerate}[label=(\roman*)]
\item if $i \leq j$ then $b_i c_j \equiv_M bc$,
\item if $i > j$ then $b_i \models q(x)|_{M c_j}$.
\end{enumerate}
\end{lemma}
\begin{proof}
We basically verify that the proof of \cite[Lemma 6.4]{kaplan_kim-independence_2020} goes through, while fixing a gap by mixing in a use of \thref{ramsey-fix} (see also \thref{rem:missing-ingredient-kaplan-ramsey}).

Let $\lambda$ be the cardinal from \thref{base-morley-tree-on-s-indiscernible-tree}, where the $C$ and $\kappa$ are $M$ and $|bc|$ respectively. Let $N \supseteq Mb$ be $\beth_\omega(|\T_\lambda|)$-saturated (note that $|\T_\lambda| \geq \lambda_T + |Mbc|$). Let $q(x, y)$ be a global $M$-Ls-invariant extension of $\Lstp(N/M)$, where the $x$ variable matches $b$. In particular, for $\beta \models q(x)$ we have $\beta \equivls_M b$. We write $p(z, b) = \tp(c/Mb)$. By induction on $\alpha \leq \lambda$ we will construct trees $(b^\alpha_\eta, c^\alpha_\eta)_{\eta \in \T_\alpha}$, such that:
\begin{enumerate}[label=(\arabic*)]
\item for all $\eta \in \T_\alpha$ we have $b^\alpha_\eta \models q(x)|_{M b^\alpha_{\rhd \eta} c^\alpha_{\rhd \eta}}$ and $b_\eta \equivls_M b$,
\item for all $\eta \in \T_\alpha$ we have $c^\alpha_\eta \models \bigcup_{\nu \unrhd \eta} p(z, b^\alpha_\nu)$,
\item the tree $(b^\alpha_\eta, c^\alpha_\eta)_{\eta \in \T_\alpha}$ is spread out and s-indiscernible over $M$,
\item for all $\beta < \alpha$ we have $b^\alpha_{\iota_{\beta \alpha}(\eta)} c^\alpha_{\iota_{\beta \alpha}(\eta)} = b^\beta_\eta c^\beta_\eta$ for all $\eta \in \T_\beta$.
\end{enumerate}
We start by setting $b^0_\emptyset c^0_\emptyset = bc$. For a limit stage $\ell$, we set $b^\ell_{\iota_{\beta \ell}(\eta)} c^\ell_{\iota_{\beta \ell}(\eta)} = b^\beta_\eta c^\beta_\eta$, where $\beta$ ranges over all ordinals $< \ell$ and $\eta$ ranges over all elements in $\T_\beta$. This is well-defined by property (4), and properties (1)--(3) then follow immediately from the induction hypothesis.

For the successor step we suppose $(b^\alpha_\eta, c^\alpha_\eta)_{\eta \in \T_\alpha}$ has been constructed. Using \thref{ramsey-fix} we find a Morley sequence $((b^\alpha_{\eta,i}, c^\alpha_{\eta,i})_{\eta \in \T_\alpha})_{i < \omega}$ in some global $M$-Ls-invariant type with $(b^\alpha_{\eta,0}, c^\alpha_{\eta,0})_{\eta \in \T_\alpha} = (b^\alpha_\eta, c^\alpha_\eta)_{\eta \in \T_\alpha}$ that is mutually s-indiscernible over $M$. Define a tree $(d_\eta, e_\eta)_{\eta \in \T_{\alpha+1}}$ by setting $d_{\langle i \rangle^\frown \eta} e_{\langle i \rangle^\frown \eta} = b^\alpha_{\eta,i} c^\alpha_{\eta,i}$ for all $\eta \in \T_\alpha$ and $i < \omega$. This leaves us to define $d_\emptyset$ and $e_\emptyset$. Let $\beta \models q(x)$ and pick $d_\emptyset$ such that
\[
d_\emptyset \equivls_{M d_{\rhd \emptyset} e_{\rhd \emptyset}} \beta.
\]
We can then apply \thref{tree-generalised-weak-independence-theorem} to the tree $(d_\eta)_{\T_{\alpha+1}}$ to see that
\[
\bigcup_{\eta \in \T_{\alpha+1}} p(z, d_\eta)
\]
does not Kim-divide over $M$. In particular, this set is consistent and so we can let $e_\emptyset$ be a realisation of this set. The $\EM_s$-type of $(d_\eta, e_\eta)_{\eta \in \T_{\alpha+1}}$ over $M$ satisfies the following properties.
\begin{enumerate}[label=(\roman*)]
\item It contains $\tp((d_\eta, e_\eta)_{\eta \in \T_{\alpha+1} \setminus \{\emptyset\}}/M)$. This is because $(d_{\rhd \langle i \rangle}, e_{\rhd \langle i \rangle})_{i < \omega}$ forms an $M$-indiscernible sequence, and by the s-indiscernibility of the individual trees in this sequence.
\item It contains the type $r(x_\emptyset, (x_\eta)_{\eta \rhd \emptyset}, (z_\eta)_{\eta \rhd \emptyset}) = \tp(d_\emptyset, d_{\rhd \emptyset}, e_{\rhd \emptyset}/M)$, and note that by construction $r(x, d_{\rhd \emptyset}, e_{\rhd \emptyset}/M) = q(x)|_{M d_{\rhd \emptyset} e_{\rhd \emptyset}}$. Indeed, let $\bar{\eta}$ and $\bar{\nu}$ be two finite tuples in $\T_{\alpha+1}$ with the same s-quantifier free type that do not contain the root. Then we have $d_{\bar{\eta}} e_{\bar{\eta}} \equivls_M d_{\bar{\nu}} e_{\bar{\nu}}$, see (i) for the justification. The claim then follows from $M$-Ls-invariance of $q$.
\item It captures that $d_\emptyset \equivls_M d_{\langle i \rangle}$ for all $i < \omega$. By construction we have $\d_M(d_\emptyset, d_{\langle 0 \rangle}) \leq n$ for some $n < \omega$, so $\d_M(d_\emptyset, d_{\langle i \rangle}) \leq n + 1$ for all $i < \omega$. By thickness $\d_M(x_\emptyset, x_{\langle i \rangle}) \leq n + 1$ is type-definable over $M$, and this partial type is thus contained in the $\EM_s$-type.
\item It captures that $e_\emptyset \models \bigcup_{\nu \unrhd \emptyset} p(z, d_\nu)$.
\end{enumerate}
We apply \thref{tree-modelling}(i) to find an s-indiscernible tree $(b^{\alpha+1}_\eta, c^{\alpha+1}_\eta)_{\eta \in \T_{\alpha+1}}$ over $M$ that is $\EM_s$-based over $M$ on $(d_\eta, e_\eta)_{\eta \in \T_{\alpha+1}}$. By an automorphism and (i) we may assume that $b^{\alpha+1}_{\langle i \rangle^\frown \eta} c^{\alpha+1}_{\langle i \rangle^\frown \eta} = d_{\langle i \rangle^\frown \eta} e_{\langle i \rangle^\frown \eta} = b^\alpha_{\eta,i} c^\alpha_{\eta,i}$ for all $\eta \in \T_\alpha$ and $i < \omega$, and so (4) is satisfied. This then also implies that (3) is satisfied. Finally, (1) is satisfied because of (ii) and (iii) and (2) is satisfied because of (iv), in both cases combined with the induction hypothesis. This completes the inductive construction.

We thus have constructed a tree $(b^\lambda_\eta, c^\lambda_\eta)_{\eta \in \T_\lambda}$ satisfying (1)--(3). We now apply \thref{base-morley-tree-on-s-indiscernible-tree} to find a Morley tree $(b_\eta, c_\eta)_{\eta \in \T_\omega}$ over $M$ str-Ls-based on $(b^\lambda_\eta, c^\lambda_\eta)_{\eta \in \T_\lambda}$ over $M$. Property (2) is clearly preserved under str-Ls-basing. To see that property (1) is preserved under str-Ls-basing we show that, for any $\eta \in \T_\omega$ and finite tuple $\bar{\nu}$ in $\T_\omega$, we have $b_\eta \models q(x)|_{M b_{\bar{\nu}} c_{\bar{\nu}}}$. Indeed, by str-Ls-basing we find $\gamma, \bar{\mu} \in \T_\omega$ such that $\gamma \bar{\mu}$ has the same str-quantifier-free type as $\eta \bar{\nu}$ and $b_\eta b_{\bar{\nu}} c_{\bar{\nu}} \equivls_M b^\lambda_\gamma b^\lambda_{\bar{\mu}} c^\lambda_{\bar{\mu}}$. Let $\beta \models q(x)$, then we have by $M$-Ls-invariance of $q(x)$ that
\[
b_\eta b_{\bar{\nu}} c_{\bar{\nu}} \equivls_M b^\lambda_\gamma b^\lambda_{\bar{\mu}} c^\lambda_{\bar{\mu}} \equiv \beta b^\lambda_{\bar{\mu}} c^\lambda_{\bar{\mu}} \equivls_M \beta b_{\bar{\nu}} c_{\bar{\nu}},
\]
as required. So setting $(b_i, c_i) = (b_{\zeta_i}, c_{\zeta_i})$ for all $i < \omega$ we find our desired tree Morley sequence.
\end{proof}
\begin{theorem}[Independence Theorem]
\thlabel{independence-theorem}
Let $T$ be a thick NSOP$_1$ theory. Suppose that $a \equivls_M a'$, $a \ind^K_M b$, $a' \ind^K_M c$ and $b \ind^K_M c$. Then there is $a''$ with $a'' \equivls_{Mb} a$, $a'' \equivls_{Mc} a'$ and $a'' \ind^K_M bc$.
\end{theorem}
\begin{proof}
We now have all the tools in place to follow the proof of \cite[Theorem 6.5]{kaplan_kim-independence_2020}. To get our conclusion about Lascar strong types, we apply the same trick as at the start of \cite[Theorem 7.7]{dobrowolski_kim-independence_2022}: as described there we may assume $b$ and $c$ to enumerate $\lambda_T$-saturated e.c.\ models containing $M$. So we have reduced our goal to proving that, for $p_0(x, b) = \tp(a/Mb)$ and $p_1(x, c) = \tp(a'/Mc)$, the partial type $p_0(x, b) \cup p_1(x, c)$ does not Kim-divide over $M$.

Let $(b_i, c_i)_{i < \omega}$ and $q(x,y)$ be as in \thref{zig-zag}, and we may assume $b_1 c_1 = bc$. Let $a''$ be such that $a'' c_0 \equivls_M a'c$, which can be done because $c = c_1 \equivls_M c_0$. We then have $a \equivls_M a''$ as well as $a \ind^K_M b_1$, $a'' \ind^K_M c_0$ and $b_1 \ind^*_M c_0$, because $b_1 \models q(x)|_{M c_0}$, so by \thref{weak-independence-theorem} we have that $p_0(x, b_1) \cup p_1(x, c_0)$ does not Kim-divide over $M$. Since $(b_i, c_i)_{i < \omega}$ is a tree Morley sequence over $M$, we can apply both parts of \thref{sub-tree-morley-sequences} to see that $(b_{2i+1}, c_{2i})_{i < \omega}$ is a tree Morley sequence over $M$. Hence by Kim's lemma for tree Morley sequences (\thref{kims-lemma-tree-morley-sequences}) we have that
\[
\bigcup_{i < \omega} p_0(x, b_{2i+1}) \cup p_1(x, c_{2i})
\]
is consistent. Thus
\[
\bigcup_{i < \omega} p_0(x, b_{2i+1}) \cup p_1(x, c_{2i+2})
\]
is consistent, as this is contained in the above set. Again, by \thref{sub-tree-morley-sequences}, we have that $(b_{2i+1}, c_{2i+2})_{i < \omega}$ is a tree Morley sequence over $M$. Since $b_1 c_2 \equiv_M bc$ we thus have by Kim's lemma for tree Morley sequences (\thref{kims-lemma-tree-morley-sequences}) again that $p_0(x, b) \cup p_1(x, c)$ does not Kim-divide over $M$, which finishes the proof.
\end{proof}

% References
\bibliographystyle{alpha}
\bibliography{bibfile}

\begin{thebibliography}{CKR23}

\bibitem[CKR23]{chernikov_transitivity_2023}
Artem Chernikov, Byunghan Kim, and Nicholas Ramsey.
\newblock Transitivity, lowness, and ranks in {NSOP}$_1$ theories, June 2023.
\newblock arXiv:2006.10486.

\bibitem[CR16]{chernikov_model-theoretic_2016}
Artem Chernikov and Nicholas Ramsey.
\newblock On model-theoretic tree properties.
\newblock {\em Journal of Mathematical Logic}, 16(02), December 2016.

\bibitem[DK22]{dobrowolski_kim-independence_2022}
Jan Dobrowolski and Mark Kamsma.
\newblock Kim-independence in positive logic.
\newblock {\em Model Theory}, 1(1):55--113, June 2022.

\bibitem[Kam23]{kamsma_positive_2023}
Mark Kamsma.
\newblock Positive indiscernibles, May 2023.
\newblock arXiv:2305.14127.

\bibitem[Kho23]{kho_primer_2023}
Wen Kho.
\newblock A {Primer} of {NSOP$_1$} {Theories}, 2023.
\newblock Master's thesis.

\bibitem[KR20]{kaplan_kim-independence_2020}
Itay Kaplan and Nicholas Ramsey.
\newblock On {Kim}-independence.
\newblock {\em Journal of the European Mathematical Society}, 22(5):1423--1474,
  January 2020.

\end{thebibliography}

% Index of all the notes
%\printindex[notes]

\end{document}